\documentclass[11pt,reqno]{amsart}

\usepackage{epigamath}

\usepackage[english]{babel}

\usepackage{amssymb}
\usepackage[all]{xy}
\usepackage{hyperref}

\usepackage{enumitem}   
\setlist[enumerate,1]{label={\rm(\thethm.\alph*)}, ref={\rm\thethm.\alph*}}

\setlist[enumerate]{itemsep=.2em,topsep=.2em,leftmargin=1.25em,itemindent=2.0em}

\newtheorem{thm}{Theorem}
\newtheorem{cor}[thm]{Corollary}
\newtheorem{claim}{Claim}

\newtheorem{lem}[thm]{Lemma}

\theoremstyle{definition}
\newtheorem{say}[thm]{ }

\newtheorem{defn-thm}[thm]{Definition--Theorem}  
\newtheorem{defn-lem}[thm]{Definition--Lemma}  

\theoremstyle{remark}
\newtheorem{exmp}[thm]{Example}
\newtheorem{rem}[thm]{Remark}
\newtheorem{ques}[thm]{Question}    
\newtheorem{subremark}[claim]{Remark}

\numberwithin{equation}{thm}

\renewcommand{\o}[0]{{\mathcal O}} 
\newcommand{\z}[0]{{\mathbb Z}}

\renewcommand{\r}[0]{{\mathbb R}} 

\renewcommand{\a}[0]{{\mathbb A}}

\newcommand{\p}[0]{{\mathbb P}}

\newcommand{\q}[0]{{\mathbb Q}}

\newcommand{\qtq}[1]{\quad\mbox{#1}\quad}

\newcommand{\pic}[0]{\operatorname{Pic}}

\newcommand{\rank}[0]{\operatorname{rank}}
\newcommand{\mult}[0]{\operatorname{mult}}

\newcommand{\supp}[0]{\operatorname{Supp}}

\newcommand{\im}[0]{\operatorname{im}}    
\newcommand{\proj}[0]{\operatorname{Proj}}

\newcommand{\trace}[0]{\operatorname{Trace}}  
\newcommand{\cent}[0]{\operatorname{center}}

\newcommand{\sing}[0]{\operatorname{Sing}}

\newcommand{\chr}[0]{\operatorname{char}}

\newcommand{\rdown}[1]{\lfloor{#1}\rfloor}

\newcommand{\simq}[0]{\sim_{\q}}

\newcommand{\depth}[0]{\operatorname{depth}} 
\newcommand{\tsum}[0]{\textstyle{\sum}}

\newcommand{\To}{\longrightarrow}

\def\into{\DOTSB\lhook\joinrel\to}

\def\loccoh#1.#2.#3.#4.{H^{#1}_{#2}(#3,#4)}

\DeclareMathAlphabet{\mathchanc}{OT1}{pzc}%
                                {m}{it}

\usepackage[all]{xy}\xyoption{dvips}

\makeatletter
\def\thmhead@plain#1#2#3{%
  \thmname{#1}\thmnumber{\@ifnotempty{#1}{ }\@upn{#2}}%
  \thmnote{. \the\thm@notefont#3}}
\let\thmhead\thmhead@plain
\makeatother

\EpigaVolumeYear{7}{2023} \EpigaArticleNr{6} \ReceivedOn{June 22, 2022}
\InFinalFormOn{October 30, 2022}
\AcceptedOn{November 14, 2022}

\title{Families of stable  3-folds in positive characteristic}
\titlemark{Stable  3-folds}

\author{J\'anos Koll\'ar}
\address{Princeton University, Princeton NJ 08544-1000, USA}
\email{kollar@math.princeton.edu}

\authormark{J.~Koll\'ar}

\AbstractInEnglish{We show that flat families of stable 3-folds do not
  lead to proper moduli spaces in any characteristic $p>0$.  As a
  byproduct, we obtain log canonical 4-fold pairs, whose log canonical
  centers are not weakly normal.}

\MSCclass{14J10, 14D22, 14G17; 14J30, 14E30} 

\KeyWords{Stable family, 3-fold, positive characteristic, canonical model, moduli}

\acknowledgement{Partial financial support was provided by the NSF
  under grant number DMS-1901855.}

\begin{document}


\maketitle

\begin{prelims}

\DisplayAbstractInEnglish

\bigskip

\DisplayKeyWords

\medskip

\DisplayMSCclass

\end{prelims}


\newpage




Following \cite{ksb}, the moduli space of varieties of general type is compactified as follows. Start with a family of smooth projective varieties $g^\circ\colon X^\circ_B\to B^\circ$ over a smooth, affine curve $B^\circ$. Possibly after a base change $C^\circ\to B^\circ$, it extends to a semi-stable family $g\colon X\to C$ over a smooth proper curve $C\supset C^\circ$, as in \cite{MR0335518}. Finally, let $g^{\rm c}\colon X^{\rm c}\to C$ be the relative canonical model as in \cite[Section~3.8]{km-book}.

In characteristic 0, the resulting fibers have semi-log-canonical singularities and ample canonical class. These are called {\it stable varieties.}  Stable varieties with a fixed volume have a coarse moduli space that is projective over $\q$; see \cite{k-modbook} for details.

Our aim is to give  examples to show that this process does not work for 3-folds in positive characteristic. For simplicity, we work over algebraically closed fields from now on.

\begin{thm}\label{no.ss.deg.charop.thm}
  Flat families of stable 3-fold pairs of fixed volume do not form a proper moduli theory in any characteristic $p>0$.
\end{thm}

\begin{rem}
 In our examples, the approach of \cite{ksb} does produce a canonical model $g^{\rm c}\colon X^{\rm c}\to C$; the problem is that some of its fibers do not satisfy Serre's condition $S_2$. However, the normalization of each fiber is a stable variety. So we may well have a proper moduli space, but we need to allow some families whose fibers are stable only after normalization. However, we do not have a precise conjecture on which families should be allowed.
  
 It is known that moduli theory is more complicated in positive characteristic. This is partly due to the failure of Kodaira's vanishing theorem, and to the current lack of resolution of singularities.  The appearance of $p$-torsion in the class group leads to other problems.  However, it was expected that once these are correctly accounted for, the rest of the arguments would go through. For surfaces, large parts of the theory have been worked out in
 \cite{pat-proj, abp-prep}.

Surprisingly, in Example~\ref{pg.jump.3d.exmp} the normalization of the central fiber does lift to characteristic 0, so Kodaira's vanishing theorem holds on it by \cite{del-ill}.  The problem comes from the nearby fibers that exist only in characteristic $p>0$.  These nearby fibers are normal, but not CM by Lemma~\ref{non.CM.Y1.lem}.  In characteristic 0, being CM is a deformation-invariant property of stable varieties; see \cite{k-db}. It is thus possible that such examples do not occur if one works solely in the closure of the moduli of smooth varieties of general type in dimensions 2 and 3; see \cite{2020arXiv200603571A, brivio-2}. As we discuss in Example~\ref{ktb.exmps}, there are such examples starting in dimension 5 in characteristic 2.

Therefore, while the examples obtained so far indicate that the positive-characteristic moduli theory is much more subtle, they do not rule out the possibility that, in the end, the necessary modifications are mainly technical.
\end{rem}

A series of non-CM singularities is discussed in the papers \cite{MR3832406, MR3994312, bern-nonsn}, but it did not seem to have been observed that they can be used to construct stable degenerations, where the generic fibers are smooth with ample canonical class, and the special fibers have isolated, non-normal singularities.  The dimension of the resulting examples is about twice the characteristic.

\begin{exmp}[(Kov\'acs--Totaro--Bernasconi examples)]\label{ktb.exmps}
  Let $X=G/P$ be a projective, homogeneous space. The cases when $P$ is non-reduced were studied in \cite{MR1247497, MR1385284, MR1453823}.  Most of these are not Fano, but if $X$ is Fano and Kodaira vanishing fails, then cones over $X$
  give interesting singularities; see  \cite{MR3832406, MR3994312, bern-nonsn}.  The smallest dimension of such an $X$ is 5 in characteristic 2, and about twice the characteristic in general.

Assume that $-K_X=mH$ for some ample divisor $H$ for some $m\geq 1$. Then  $|H|$ is very ample by \cite{MR1385284}, so it gives an embedding $X\into \p^N$, where $N=\dim |H|$. Let $Y:=C(X, H)\subset \p^{N+1}$ be the projective cone over $X$ with vertex $v$. Then
   $$
   H^{i+1}_v(Y, \o_Y)\cong \tsum_{m\in\z} H^i\bigl(X, \o_X(mH)\bigr)
   $$
by \cite[Remark~3.12]{kk-singbook}.  If $H^1\bigl(X, \o_X(H)\bigr)\neq 0$, then $H^{2}_v(Y, \o_Y)\neq 0 $, hence $\depth_v\o_Y=2$.

Let $D\in |H|$ be a smooth divisor and $D_Y\subset Y$ its preimage. Since $K_X+ D\sim (m-1)H$, the pair $(Y, D_Y)$ is log canonical if $m=1$ and canonical  if $m>1$ by \cite[Section~3.1]{kk-singbook}. The divisor $D_Y$ is Cartier on $Y$ by \cite[Proposition~3.14(2)]{kk-singbook}.

There is a natural morphism $\pi\colon C(D, H|_D)\to D_Y$, which is an isomorphism outside the vertex. Furthermore, $\pi$ is an isomorphism if and only if $\depth_vD_Y\geq 2$, which holds if and only if $\depth_vY\geq 3$. Thus, if $H^1\bigl(X, \o_X(H)\bigr)\neq 0$, then $D_Y$ is not normal.  Intersecting $Y$ with a pencil of hyperplanes gives a locally stable degeneration with generic fiber $X$ and special fiber $D_Y$.  Taking a suitable cyclic cover as in \cite[Section~2.4]{km-book}, we get a series of examples of stable degenerations, where the generic fibers are smooth varieties with ample canonical class, and the special fibers have isolated non-normal singularities.

The cases described in \cite{MR3994312} have $m= 2$. Then the normalization of $D_Y$ has canonical singularities; hence these deformations take place in what is usually considered the `interior' of the moduli space.  On the other hand, these constructions start with a variety for which Kodaira vanishing fails, so they cannot be lifted to characteristic 0.
   
{\it Aside.} The series of non-CM, quotient singularities of \cite{MR3929517} all have $\depth\geq 3$ by \cite{MR581583}, so they do not give similar examples.
\end{exmp}

All the ingredients going into the proof of Theorem~\ref{no.ss.deg.charop.thm} seem to be well known. Unipotent bundles on elliptic curves have been studied in \cite{Atiyah57, MR0292847, oda-ell}.  Closely related examples of pathological families of elliptic surfaces are given in \cite{brivio}.  Using cones to go from a lower-dimensional non-general type variety to a higher-dimensional general type one has been utilized many times.  For example, \cite{mck-lec-2006} used it to show that log abundance for general type varieties in dimension $n+1$ implies log abundance for all varieties in dimension $n$; see Section~\ref{mck.t.say}.  Related results on the deformation invariance of plurigenera also appear  in \cite{MR4298650, be-bi-st}.

\medskip

The main step is the following; see \cite[Definition~2.37]{km-book}
for the definition of {\it divisorial log terminal,} abbreviated as {\it dlt.}

\begin{exmp}\label{pg.jump.3d.exmp}
In every characteristic $p>0$, there exist a morphism $g\colon Y\to \a^1$ and an effective $\q$-divisor $\Delta$ on $Y$ such that
  \begin{enumerate}
\item\label{pg.jump.3d.exmp-1} $g\colon Y\to \a^1$ is smooth, projective, of relative dimension 3;
\item\label{pg.jump.3d.exmp-2} $K_Y+\Delta$ is $g$-semiample and $g$-big;
\item\label{pg.jump.3d.exmp-3} the fibers  $(Y_t, \Delta_t)$ are dlt;
  \item\label{pg.jump.3d.exmp-4} $(Y, Y_t+\Delta)$ is dlt for every $t\in \a^1$, and this continues to hold after any base change $C\to \a^1$;
\item\label{pg.jump.3d.exmp-5} $H^0\bigl(Y_0, \omega^{m}_{Y_0}(m\Delta_0)\bigr)> H^0\bigl(Y_t, \omega^{m}_{Y_t}(m\Delta_t)\bigr)$ for $t\neq 0$ and $ m\geq 1$ sufficiently divisible.
\end{enumerate}
\end{exmp}

As in \cite[Theorem~4.9]{kk-singbook},
we have \eqref{pg.jump.3d.exmp-4} $\Rightarrow$ \eqref{pg.jump.3d.exmp-3} by adjunction. In characteristic 0, inversion of adjunction says that \eqref{pg.jump.3d.exmp-3} implies the log canonical variant of \eqref{pg.jump.3d.exmp-4}. The key new feature is the jump of the plurigenera \eqref{pg.jump.3d.exmp-5}.

There are many examples in positive characteristic where finitely many of the plurigenera jump.  The surfaces produced in \cite{brivio} lead to families of elliptic pairs (with terminal singularities) where infinitely many of the plurigenera jump.  These surfaces can also be used (instead of Example~\ref{pg.jump.2d.exmp}) to obtain families of 3-folds as in Example~\ref{pg.jump.3d.exmp}.

The new feature of Example~\ref{pg.jump.3d.exmp} is that the fibers are of general type, and infinitely many of the plurigenera jump.  This has strong consequences.

\begin{proof}[Proof of Theorem~\ref{no.ss.deg.charop.thm}] We deduce the theorem from Example~\ref{pg.jump.3d.exmp}. 
 Since $K_Y+\Delta$ is $g$-semiample and $g$-big, there exist a morphism with connected fibers $h\colon Y\to Y'$ and a relatively ample divisor $D'$ on $Y'\to \a^1$ such that $m_0(K_Y+\Delta)\sim h^*(m_0D')$ for some $m_0>0$.  By definition, the relative canonical model is
  $$
  \bigl(Y^{\rm c}, \Delta^{\rm c}\bigr):=\proj_{\a^1}\oplus_{m\geq 0}
  g_*\omega^{m}_{Y}(\rdown{m\Delta}).
  $$
On the right-hand side, we can change the summation to multiples of $m_0$, which shows that $\bigl(Y^{\rm c}, \Delta^{\rm c}\bigr)=\bigl(Y', h_*\Delta\bigr)$.  The fiber over the origin is
$$
  \bigl(Y^{\rm c}, \Delta^{\rm c}\bigr)_0=\proj\oplus_{m\geq 0}
  \im\bigl[g_*\omega^{mm_0}_{Y}(mm_0\Delta)\to H^0\bigl(Y_0, \omega^{mm_0}_{Y_0}(mm_0\Delta_0)\bigr)\bigr],
  $$
whereas the canonical model of the   fiber  $(Y_0, \Delta_0)$ is
  $$
  \bigl((Y_0)^{\rm c}, (\Delta_0)^{\rm c}\bigr):=\proj\oplus_{m\geq 0}
   H^0\bigl(Y_0, \omega^{mm_0}_{Y_0}(mm_0\Delta_0)\bigr).
   $$
By \eqref{pg.jump.3d.exmp-5}, we see that
   $$
   \bigl(Y^{\rm c}, \Delta^{\rm c}\bigr)_0\neq \bigl((Y_0)^{\rm c}, (\Delta_0)^{\rm c}\bigr).
   $$
   The properties \eqref{pg.jump.3d.exmp-1}--\eqref{pg.jump.3d.exmp-5} continue to hold after any base change $C\to \a^1$, so $\bigl(Y^{\rm c}, \Delta^{\rm c}\bigr)\times_{\a^1}C\to C $ is the canonical model of $\bigl(Y, \Delta\bigr)\times_{\a^1}C\to C $.  Thus $\bigl(Y^{\rm c}, \Delta^{\rm c}\bigr)_0$ is the unique stable degeneration of the family over $\a^1\setminus\{0\}$ by \cite[Theorem~11.40]{k-modbook}.
\end{proof}

In our examples, $\sing (Y_t)^{\rm c}$ is 1-dimensional. Localizing at its generic point gives a simple elliptic singularity of dimension 2.  There is a natural morphism
$$
\bigl((Y_0)^{\rm c}, (\Delta_0)^{\rm c}\bigr)\To \bigl(Y^{\rm c}, \Delta^{\rm c}\bigr)_0,
$$
which is an isomorphism outside the singular set and purely inseparable over $\sing (Y^{\rm c})_0$.  These imply the following; see \cite[Section~7.2]{rc-book} for {\it weak normality}.

\begin{cor}\label{non.wn.4.cor}
For any $p>0$, there are log canonical 4-fold pairs $(X, S+\Delta)$ of characteristic $p$, where $S$ is a Cartier divisor that is not weakly normal.
\end{cor}

\begin{proof}
The pair $\bigl(Y^{\rm c}, \Delta^{\rm c}+(Y^{\rm c})_0\bigr)$ is a log canonical 4-fold, $(Y^{\rm c})_0$ is a Cartier divisor and a log canonical center but not weakly normal.
\end{proof}

\begin{rem}
In characteristic 0, weakly normal coincides with seminormal, and all log canonical centers are seminormal by \cite{ambro, fuj-book}; see also \cite[Theorem~5.14]{kk-singbook}.  A series of papers culminating in \cite{many-p} establish that MMP for 3-folds works in characteristics at least $7$, just as in characteristic 0. This led to a hope that, in any fixed dimension, new phenomena appear only in low characteristics; see \cite{MR3994312, bern-nonsn, h-w-1, h-w-p} for such results.  By Corollary~\ref{non.wn.4.cor}, this is not the case for 4-folds.
  
However, in MMP one can frequently work with dlt pairs only; thus a dlt example would be the main test case.  Our $\bigl(Y^{\rm c}, \Delta^{\rm c}\bigr)$ are not dlt; see Lemma~\ref{non.CM.Y1.lem}.

Let $(X, S)$ be a 3-dimensional plt pair. If $\chr\geq 7$, then $S$ is normal, but there are counterexamples in characteristics 2; see \cite{cas-tan-dp2}. See also \cite{ber-dp3, lac-dp, 2020arXiv200603571A} for related examples.  However, no such example seems to be known where $S$ is Cartier.
\end{rem}

\begin{ques}\label{3d.dn.ques}
Let $(X, S+\Delta)$ be a 3-dimensional, semi-log-canonical pair, where $S$ is a Cartier divisor. Is $S$ weakly normal?
  \end{ques}

We start the construction of our 3-dimensional examples with a family of elliptic surfaces, similarly to \cite{brivio}.  Then we build these up to dimension 3.

\begin{exmp}\label{pg.jump.2d.exmp}
  We construct a morphism $g\colon(S, \Delta+ \Theta)\to \a^1$ of relative dimension 2, where $g\colon S\to \a^1$ is smooth, projective, the fibers $(S_t, \Delta_t+\Theta_t)$ are terminal, $K_S+\Delta$ is $g$-semiample, numerically trivial, and $K_S+\Delta+ \Theta$ is $g$-semiample of Kodaira dimension 1.  Moreover, all large plurigenera jump:
$$
H^0\bigl(S_0, \omega^{m}_{S_0}(\rdown{m\Delta_0+m\Theta_0})\bigr)>
H^0\bigl(S_t, \omega^{m}_{S_t}(\rdown{m\Delta_t+m\Theta_t})\bigr)
\quad \text{for}\; t\neq 0, m\gg 1.
$$
\end{exmp}

The key is the behavior of unipotent vector bundles on elliptic curves in positive characteristic.

\begin{say}[{\bf Elliptic ruled surfaces}]\label{ell.ruled.say}
Let $E$ be an elliptic curve. A vector bundle is {\it unipotent} if it can be written as a successive extension of copies of $\o_E$.

By \cite{Atiyah57}, for every $r$, there is a unique rank $r$, indecomposable, unipotent vector bundle $F_r=F_r(E)$, and every unipotent vector bundle is a direct sum of these $F_r$.  Thus a unipotent vector bundle $U$ is isomorphic to $F_r$ if and only if $\rank U=r$ and $h^0(E, U)=1$.

Note that $F_2(E)$ sits in an exact sequence
  $$
  0\To \o_E\To F_2(E)\To \o_E\To 0
  $$
that corresponds to a non-zero class in $H^1(E, \o_E)$.

Let $\tau\colon E'\to E$ be a non-constant map of elliptic curves. Then we get
  $$
  0\To \o_{E'}\To \tau^*F_2(E)\To \o_{E'}\To 0.
  $$
This shows that
\begin{itemize}
  \item if $\tau^*\colon H^1(E, \o_E)\to H^1(E', \o_{E'}) $ is non-zero, then $\tau^*F_2(E)\cong F_2(E')$;
 \item  if $\tau^*\colon H^1(E, \o_E)\to H^1(E', \o_{E'}) $ is zero, then
  $\tau^*F_2(E)\cong \o_{E'}\oplus \o_{E'}$. This holds if and only if $h^0\bigl(E',\tau^*F_2(E)\bigr)\geq 2$.
\end{itemize}

Set $m=\deg \tau$.  If the characteristic does not divide $m$, then $\frac1{m}\trace_{E'/E}\colon \tau_*\o_{E'}\to \o_E$ is a splitting of $\o_E\to \tau_*\o_{E'}$. Thus $ H^1(E, \o_E)\to H^1(E', \o_{E'}) $ is an injection. In particular, we are always in the first case in characteristic 0.
  
  Note that $\tau^*\colon H^1(E, \o_E)\to H^1(E', \o_{E'})$ is the tangent map of $g^*\colon \pic(E)\to \pic(E')$ at the origin. (This holds for Abelian varieties; see \cite[Section~15]{mumf-abvar}.) Thus $\tau^*\colon H^1(E, \o_E)\to H^1(E', \o_{E'})$ is the zero map if and only if $\tau^*\colon \pic(E)\to \pic(E')$ is inseparable. This cannot happen in characteristic 0, but in characteristic $p>0$, there is always a degree $p$ map $ \rho\colon E'\to E$ such that $\rho^*\colon H^1(E, \o_E)\to H^1(E', \o_{E'})$ is the zero map.  (For higher-dimensional Abelian varieties, $\rho^*$ has a 1-dimensional kernel; see \cite[Section~15]{mumf-abvar}.)

Projectivising $F_2$, we get a $\p^1$-bundle $\pi_1\colon S_1\to E$ with a unique section $D_1$ with self-intersection 0. Note that $K_{S_1}+2D_1\sim 0$ and the normal bundle of $D_1$ is trivial.

Let $C\subset S_1$ be an irreducible, reduced curve that is disjoint from $D_1$. Then it is numerically equivalent to a multiple of $D_1$; hence $p_a(C)=1$ by adjunction. The projection $\tau\colon C\to E$ is dominant; hence $C$ is a smooth, elliptic curve.  The fiber product $S_1\times_EC$ has two disjoint section; hence $\tau^*F_2(E)$ is trivial.  Thus $H^1(E, \o_E)\to H^1(C, \o_C)$ is the zero map.

Conversely, let $g\colon C\to E$ be a non-constant morphism such that $g^*\colon H^1(E, \o_E)\to H^1(C, \o_C)$ is the zero map. Then $g^*F_2$ splits, giving $C\to S_1$ whose image is linearly equivalent to $\deg(g)\cdot D_1$.

\begin{claim}\label{ell.ruled.say-1}
We have $(\pi_1)_*\o_{S_1}(mD_1)=F_{m+1}$  if either $\chr k=0$ or $m<\chr k$.
\end{claim}
  
\begin{proof}
  Let $C$ be an integral curve such that $C\sim mD_1$.  Then $(C\cdot D_1)=0$, so either $C$ and $D_1$ are equal, or they are disjoint. As we noted above, the latter cannot happen if either $\chr k=0$ or $m<\chr k$.  Thus $h^0(S_1, \o_{S_1}(mD_1))=1$ for such values. Pushing forward the filtration
$$
\o_{S_1} \subset \o_{S_1}(D_1) \subset \cdots \subset \o_{S_1}(mD_1)
$$
gives that  $(\pi_1)_*\o_{S_1}(mD_1)$ is unipotent.
\end{proof}

\begin{claim}\label{ell.ruled.say-2}
  If $\chr k=p>0$, then $(\pi_1)_*\o_{S_1}(pD_1)=\o_E\oplus F_{p}$.
\end{claim}

\begin{proof}
Choose a degree $p$ map $ \rho\colon E'\to E$ such that $\rho^*\colon H^1(E, \o_E)\to H^1(E', \o_{E'})$ is zero.  Then pull-back by $\rho\colon E'\to E$ splits $F_2$, so $H^0(S_1, \o_{S_1}(pD_1))\geq 2$. On the other hand, from
$$
0\To  \o_{S_1}((p-1)D_1) \To  \o_{S_1}((pD_1)\To \o_{D_1}\To 0
$$
we get that $F_p\into (\pi_1)_*\o_{S_1}(pD_1)$ and $H^0(S_1, \o_{S_1}(pD_1))\leq 2$. These imply that we have $ (\pi_1)_*\o_{S_1}(pD_1)\bigr)\cong \o_E\oplus F_p$.
\end{proof}

It is not clear what the $(\pi_1)_*\o_{S_1}(mD_1)$ are for larger values of $m$, but we have the following.

\begin{claim}\label{ell.ruled.say-3}
In all cases, $h^0(S_1, \o_{S_1}(mD_1))<m+1$ for $m\geq 1$.
\end{claim}

\begin{proof}
The filtration
$$
\o_{S_1}((p-1)D_1) \subset \o_{S_1}(pD_1) \subset \cdots \subset \o_{S_1}(mD_1)
$$
gives that $H^0(S_1, \o_{S_1}(mD_1))\leq 1+(m-p+1)=m+1-(p-1)<m+1$.
\end{proof}

Note that if $\chr k=0$, then in fact $h^0(S_1, \o_{S_1}(mD_1))=1$.  If $\chr k=p>0$, then we expect that $h^0(S_1, \o_{S_1}(mD_1))=\rdown{m/p}+1$.  (See \cite{schroer} for computing $F_m\otimes F_n$.)

\begin{subremark}\label{ell.ruled.say-4}
The reader can check that in our constructions, one can replace $E$ with any Abelian variety $A$ of dimension $n$. The key property is that there is a non-split extension
$$
0\To \o_A\To F_2(A)\To \o_A\To 0
$$
that becomes split after a suitable degree $p$ cover $A'\to A$.  By \cite[Section~15]{mumf-abvar}, this always holds in characteristic $p>0$.  The resulting $g\colon (Y, \Delta)\to \a^1$ has relative dimension $n+2$.  The singular sets of the canonical models $(Y_t)^{\rm c}$ are still 1-dimensional.
\end{subremark}
\end{say}

\begin{say}[\textbf{Construction of Example~\ref{pg.jump.2d.exmp}}]\label{pg.jump.2d.exmp.c}
Over $\a^1$, there is a degeneration of $F_2$ to the split bundle $\o_E\oplus \o_E$. Thus we get a vector bundle $F$ over $E\times \a^1$ whose restriction to $E\times \{0\}$ is $\o_E\oplus \o_E$ and to $E\times \{t\}$ is $F_2$ for $t\neq 0$.

Let $\pi\colon S\to E\times \a^1$ be the corresponding $\p^1$-bundle with the unique section $D$ in $|\o_{S}(1)|$.  If $0\neq t\in \a^1$, then we get $\pi_t\colon (S_t, D_t)\to E_t=E$ isomorphic to $\pi_1\colon (S_1, D_1)\to E$, and over $0\in\a^1$ we get $(S_0, D_0)\cong (E\times \p^1, E\times\{\infty\})$.

The jump in the plurigenera comes from Claim~\ref{ell.ruled.say-3}:
$$
H^0\bigl(S_0, \o_{S_0}(mD_0)\bigr)=m+1>H^0\bigl(S_1, \o_{S_1}(mD_1)\bigr).
$$
Now  consider $K_S+3D=(K_S+2D)+D\sim D$. It is nef, and by the above computations, all the plurigenera of the fibers jump at $c=0$.  However, $(S_t, 2D_t)$ is not log canonical, so this is not yet very useful.

\medskip

Now assume  that we are in characteristic $p>0$, and let $\rho\colon E'\to E$ be a degree $p$ morphism as in Section~\ref{ell.ruled.say} such that $\rho^*\colon H^1(E, \o_E)\to H^1(E', \o_{E'})$ is the zero map.

Consider $(\rho,1)\colon E'\times \a^1\to E\times \a^1$. The pull-back of $S_t$ is a trivial $\p^1$-bundle for every $t\in \a^1$, so $(\rho,1)^*S$ is a trivial $\p^1$-bundle over $E'\times \a^1$. Pushing forward, we get that $|pD|$ is a basepoint-free relative pencil on~$S$.

We can thus replace $2D$ with a linear combination
$$
\Delta:= \tfrac1{mp}G_1+\cdots + \tfrac1{mp}G_{2m}, 
$$
where the $G_i\in |pD|$ are general and $m$ is large enough.  We then have a pair $(S,\Delta)$ (we can even make the fibers terminal) such that $K_S+\Delta\simq 0$. Now we have a locally stable morphism
\begin{equation}\label{pg.jump.2d.exmp.c-1}
\bigl(S, \Delta+D\bigr)\to \a^1,
\end{equation}
where $K_S+\Delta+D$ is semiample, and all (large enough) plurigenera jump at $t=0$.  (If we want terminal fibers, we can also replace $D$ with $\tfrac1{mp}G'_1+\cdots + \tfrac1{mp}G'_{m} $ as above.)

Working on the surface $S_0$, the Iitaka fibration of $K_{S_0}+\Delta_0+D_0$ is the coordinate projection $\tau_0\colon  S_0\cong E\times \p^1\to \p^1$.  However, the restriction of the (relative) Iitaka fibration of $S$ to $S_0$ is given by a pencil in $|p D_0|$ whose members are geometrically connected.  The only such pencil is the composite of $\tau_0$ with the Frobenius morphism $\p^1\to \p^1$.
\end{say}
\medskip

This example has Kodaira dimension 1, but we can use it to get a general type
example in one dimension higher, using the following method.

\begin{say}\label{mck.t.say}
Let $Z$ be a projective variety whose pluricanonical maps behave in unexpected ways. Then \cite{mck-lec-2006} constructs a general type pair (of dimension one higher) whose pluricanonical maps also behave unexpectedly.

Choose a sufficiently general embedding $Z\into \p^N$, and let $C\subset \p^{N+1}$ be the cone over $Z$ with vertex $v\in C$.

Blow up $v$ to get $Y\to C$ with exceptional divisor $E\cong Z$.  By adjunction, $ \omega_Y(E)|_E\cong \omega_E$. So, {\em assuming} that the maps
\begin{equation}\label{mck.t.say-1}
  H^0\bigl(Y, \omega^m_Y(mE)\bigr)\To H^0(E, \omega^m_E)\cong H^0(Z, \omega^m_Z)
\end{equation} 
are surjective, the behavior of the pluricanonical maps of $Z$ should be `visible' from the pluricanonical maps of the pair $(Y, E)$.

The problem is that $Y$ is a $\p^1$-bundle over $Z$, so in fact $H^0\bigl(Y, \omega^m_Y(mE)\bigr)=0$ for $m>0$.  There are two ways to fix this.  Let $H_C$ be the hyperplane class on $C$ and $H_Y$ its pull-back to $Y$.

First, choose $d\gg 1$, and let $D_C\sim d H_C$ be a general divisor not passing through $v$. It gives $D_Y\subset Y$, with $D_Y$ disjoint from $E$, and the pair $(Y, D_Y+E)$ is of general type. We work out below that the maps
\begin{equation}\label{mck.t.say-2}
  H^0\bigl(Y, \omega^m_Y(mD_Y+mE)\bigr)\To H^0(E, \omega^m_E)\cong H^0(Z, \omega^m_Z)
\end{equation}
are surjective.

Second, fix $r>0$ not divisible by the characteristic, fix $d\gg 1$, and let $s\in \o_{C}(drH_C)$ be a general section. As in \cite[Section~2.4]{km-book}, these data determine degree $r$, ramified, cyclic covers
  $$
  \pi_C:C_r:=C\bigl[\sqrt[r]{s}\bigr]\To C \qtq{and}
  \pi_Y:Y_r:=Y\bigl[\sqrt[r]{s}\bigr]\To Y,
  $$
  whose canonical classes are $\pi_C^*\bigl(K_C+d(r-1)H_C\bigr)$ (resp.\ $\pi_Y^*\bigl(K_Y+d(r-1)H_Y\bigr)$).

  Here $E_r:=\pi_Y^{-1}(E)$ is $r$ disjoint copies of $E$.  We check that the maps
 \begin{equation}\label{mck.t.say-3}
  H^0\bigl(Y_r, \omega^m_{Y_r}(mE_r)\bigr)\To H^0(E_r, \omega^m_{E_r})\cong \oplus_1^rH^0(Z, \omega^m_Z)
\end{equation}
  are surjective.
\end{say}
\medskip

Next we go through the details of this for log canonical pairs $(Z, \Delta_Z)$ such that $K_Z+\Delta_Z\simq 0$.  For these, the cone is also log canonical, and the vertex is a log canonical center.

\begin{say}[\textbf{Cone construction}]\label{cone.contr.say}
  Start with a projective pair $(Z, \Delta_Z)$ such that $K_Z+\Delta_Z\simq 0$, and a semiample $\q$-divisor $D_Z$. Choose any ample $L$, and consider $Y:=\proj_Z(\o_Z+L)$. It is a $\p^1$-bundle $\tau\colon Y\to Z$ with two sections: $Z_0\subset Y$ with normal bundle $L^{-1}$, and $Z_\infty\subset Y$ with normal bundle $L$. With $\Delta_Y:=\tau^*\Delta_Z$, we have
$$
K_Y+\Delta_Y+Z_0+Z_\infty\simq 0.
$$
Thus $K_Y+ \Delta_Y+Z_0+2Z_\infty+D_Y$ is semiample and big.  We get its canonical model $\phi_Y\colon Y\to Y^{\rm c}$. Here $\phi_Y$ is an isomorphism on $Y\setminus Z_0$, and it restricts to the morphism given by multiples of $D_Z$ on $Z_0\cong Z$.

If $(Z, \Delta_Z+D_Z)$ is dlt, then so is $(Y, \Delta_Y+Z_0+Z_\infty+D_Y)$ by \cite[Proposition~2.15]{kk-singbook}.  Since $|Z_\infty|$ is ample on $Y\setminus Z_0$, we can replace $2Z_\infty$ with a suitable $\q$-divisor $Z'_\infty\simq 2Z_\infty$ to get a dlt pair
\begin{equation}\label{cone.contr.say-1}
(Y,\Delta_Y+Z_0+Z'_\infty+D_Y)
\end{equation}
with these properties.

If $L$ is very ample, then we can let $Z'_\infty$ be the sum of two general members of $|Z_\infty|$; otherwise, we may need $Z'_\infty$ to be a $\q$-divisor.  This may be preferable anyhow since this way we can arrange that $(Y,\Delta_Y+Z_0+Z'_\infty+D_Y)$ is terminal away from $Z_0$.  Note, however, that $Z_0$ must stay with coefficient 1, so dlt is the best that we can have for $(Y,\Delta_Y+Z_0+Z'_\infty+D_Y)$.  This is in accordance with \cite[Corollary~1.3]{2020arXiv200603571A}.
\end{say}

\begin{say}[\textbf{Computing the plurigenera}]\label{coh.comp.on.Y}
Let $(Y,\Delta_Y+Z_0+Z'_\infty+D_Y)$ be as above.  Taking into account that $K_Y+\Delta_Y+Z_0+Z_\infty\simq 0$, for all sufficiently divisible $m$, we have
\begin{equation}\label{coh.comp.on.Y-1}
\begin{split}
 H^0\bigl(Y, \o_Y(mK_Y+m\Delta_Y+mZ_0+mZ'_{\infty}+mD_Y)\bigr)&=
H^0\bigl(Y, \o_Y(mZ_{\infty}+mD_Y)\bigr)\\
&=
H^0\bigl(Z, \o_Z(mD_Z)\otimes S^m(\o_Z+L)\bigr)\\
&= \tsum_{r=0}^m H^0\bigl(Z, L^r(mD_Z)\bigr).
\end{split}
\end{equation}
Since $D_Z$ is semiample, by Fujita's vanishing theorem \cite{MR726440} we can choose $L$ sufficiently ample such that
$$
h^0\bigl(Z, L^r(mD_Z)\bigr)=\chi\bigl(Z, L^r(mD_Z)\bigr)
$$
for $r\geq 1$ and for all sufficiently divisible $m$. These terms are thus deformation invariant.  (In our case, this actually works for any $L$ since Kodaira vanishing holds on the $S_t$; \textit{cf.} \cite[Theorem~3]{MR3079312}.)

The key term is the $r=0$ summand $H^0\bigl(Z, \o_Z(mD_Z)\bigr)$ of \eqref{coh.comp.on.Y-1}.  Thus we see that if we vary the pair $(Z, \Delta_Z+D_Z)$, then any jump in the plurigenera of $(Z, \Delta_Z+D_Z)$ leads to the same jump in the plurigenera of $(Y,\Delta_Y+Z_0+Z'_\infty+D_Y)$, for all sufficiently divisible $m$.
\end{say}

\begin{say}[\textbf{Construction of Example~\ref{pg.jump.3d.exmp}}]
Start with $\bigl(S, \Delta+\tfrac12D\bigr)\to \a^1$ as in \eqref{pg.jump.2d.exmp.c-1}. Let $L$ be any relatively ample line bundle on $S$, and set $Y:=\proj_S(\o_S+L)$.  We get $(Y,\Delta_Y+Z_0+Z'_\infty+D_Y)\to \a^1$ as in \eqref{cone.contr.say-1}.

Claims \eqref{pg.jump.3d.exmp-3}--\eqref{pg.jump.3d.exmp-4} follow from Corollary~\ref{mult.1.bch.cor}, and the jump of the plurigenera \eqref{pg.jump.3d.exmp-5} is computed in Section~\ref{coh.comp.on.Y}. 
\end{say}

In characteristic 0 one usually proves claims \eqref{pg.jump.3d.exmp-3}--\eqref{pg.jump.3d.exmp-4} using inversion of adjunction as in \cite[Theorem~4.9]{kk-singbook}.
This is not known in positive characteristic; see however \cite{ha-xu-pos}
for related results
for 3-folds.  The following is a combination of \cite[Exercise~6.18]{ksc}
and \cite[Proposition~2.7]{kk-singbook}.

   \begin{lem}\label{mult.1.lc.lem}
     Let $X$ be a regular scheme and $\sum_{i\in I} D_i$ a simple normal crossing divisor. For $J\subset I$, set $D_J:=\cap_{i\in J}D_i$.  Let $\Delta$ be an effective $\r$-divisor. Assume that
 \begin{enumerate}
   \item\label{mult.1.lc.lem-1} for every $J$, none of the irreducible components of\, $D_J$ is contained in $\supp \Delta$, and $\mult_x(\Delta|_{D_J})\leq 1$ $($resp.\ $<1)$ for every $x\in D_J$.
     \end{enumerate}
Then $(X, D+\Delta)$ is log canonical $($resp.\ dlt\,$)$ near $D$.
   \end{lem}

   \begin{proof}
     Let $E$ be an exceptional divisor over $X$. In order to show that $a(E, X, D+\Delta)\geq -1$, we may localize at the generic point of $\cent_XE\subset X$.  So assume that $ \cent_XE=\{x\}$ is a closed point.  If $x\notin \supp \Delta$, then we are done by \cite[Proposition~2.7]{kk-singbook}.  Otherwise, let $\pi\colon B_xX\to X$ be the blow-up with exceptional divisor $D_0$.  Using induction and \cite[Lemma~2.45]{km-book},
it is enough to prove the following two claims:
   \begin{enumerate}\setcounter{enumi}{1}
   \item $a(D_0, X, D+\Delta)\geq -1$  (resp.\ $>-1$), and
   \item $\bigl(B_xX, (D_0+\pi^{-1}_*D)+\pi^{-1}_*\Delta\bigr)$ also satisfies
     \eqref{mult.1.lc.lem-1}.
   \end{enumerate}
   These are both straightforward.
   \end{proof}
   
   \begin{cor} \label{mult.1.bch.cor}
Let $g\colon  (X, D+\Delta)\to C$ be a smooth morphism.  Assume that $\bigl(X, (X_c+D)+\Delta\bigr)$ satisfies \eqref{mult.1.lc.lem-1} for every $c\in C$.

Then $\bigl(X, (X_c+D)+\Delta\bigr)$ is log canonical $($resp.\ dlt\,$)$ for every $c\in C$, and the same holds after every base change $C'\to C$.
     \end{cor}


The following also shows that the canonical models of the fibers do not form a flat family.

\begin{lem}\label{non.CM.Y1.lem}
The $(Y_i)^{\rm c}$ are not dlt, $(Y_0)^{\rm c}$ is CM, but $(Y_1)^{\rm c}\cong (Y^{\rm c})_t$ for $t\neq 0$ is not CM.
\end{lem}

\begin{proof}
We have $\tau_{Y_i}\colon Y_i\to (Y_i)^{\rm c}$, which restricts to the elliptic fibration $\tau_i\colon  S_i\to \p^1$. The localization of $(Y_i)^{\rm c}$ at the generic point of its singular set is thus a 2-dimensional, elliptic singularity. In particular, the $(Y_i)^{\rm c}$ are not dlt.

For the CM claims, we aim to use Lemma~\ref{nonCM.and.R1.lem}. As in Section~\ref{coh.comp.on.Y}, the cohomologies of line bundles on the~$Y_i$ are easy to compute.  In particular, we obtain that $R^1(\tau_{Y_i})_*\o_{Y_i}=R^1(\tau_{i})_*\o_{S_i}$.

Since $S_0\to \p^1$ is trivial, $ R^1(\tau_{0})_*\o_{S_0}\cong \o_{\p^1}$.  Next we show that $R^1(\tau_{1})_*\o_{S_1}$ has a 0-dimensional associated point.

With $D_1$ as in Section~\ref{ell.ruled.say}, $\tau_1$ is given by the pencil $|pD_1, C|$. The general member is a smooth, elliptic curve, so $R^1(\tau_{1})_*\o_{S_1}$ is a coherent sheaf of generic rank 1.  If it is torsion-free, then it is a line bundle. In particular, $H^1(pD_1, \o_{pD_1})=1$.  Next we compute that $H^1(pD_1, \o_{pD_1})=2$, giving a contradiction, so $R^1(\tau_{1})_*\o_{S_1}$ is not torsion-free.

Consider the exact sequence
$$
0\To \o_{S_1}\To \o_{S_1}(pD_1)\To \o_{pD_1}\To 0.
$$
Pushing forward to $E$ and using Claim~\ref{ell.ruled.say-2},  we get
$$
0\To \o_E\stackrel{\alpha}{\longrightarrow} F_p\oplus \o_E\To \pi_*\o_{pD_1}\To 0.
$$
Note that $\o_{S_1}\to \o_{S_1}(pD_1)$ factors through $\o_{S_1}((p-1)D_1)$, and $\pi_*\o_{S_1}((p-1)D_1)$ gives that $F_p$ summand.  Thus the composite of $\alpha$ with the second projection $\alpha_E\colon \o_E\to \o_E\oplus F_p\to \o_E $ is zero.  Therefore, $\pi_*\o_{pD_1}\cong F_{p-1}\oplus \o_E$, and so $H^1(pD_1, \o_{pD_1})=2$.
\end{proof}

\begin{lem}\label{nonCM.and.R1.lem}
Let $g\colon Y\to X$ be a birational morphism of normal, projective 3-folds.  Let $L$ be an ample line bundle on $X$.  Assume that $H^i(Y, g^*L^{-r})=0$ for $i<3$ and $r\gg 1$.  Then $X$ is CM if and only if $R^1g_*\o_Y$ does not have a 0-dimensional associated point.
\end{lem}

\begin{proof}
The Leray spectral sequence computing $H^i(Y, g^*L^{-r})$ contains
$$
  H^1(X, L^{-r})\lhook\joinrel\To H^1(Y, g^*L^{-r}) \quad\text{and}\quad
  H^0(X, L^{-r}\otimes R^1g_*\o_Y)\To H^2(X, L^{-r})\To H^2(Y, g^*L^{-r}).
  $$
Thus $H^1(X, L^{-r})=0$, and $H^2(X, L^{-r})=0$ for $r\gg 1$ if and only if $R^1g_*\o_Y$ has no 0-dimensional associated points. The rest follows from \cite[Corollary~5.72]{km-book}.
\end{proof}
   
\subsection*{Acknowledgments}
I thank F.~Bernasconi, I.~Brivio, P.~Cascini, C.~Hacon, S.~Ko\-v\'acs, J.~M\textsuperscript{c}Kernan, Zs.~Patakfalvi and C.~Xu for suggestions and references, and B.~Totaro for correcting earlier mistakes in Section~\ref{ell.ruled.say}.


\newcommand{\etalchar}[1]{$^{#1}$}

\end{document}